\documentclass{article} 

\usepackage{amsmath}  
\usepackage{amsthm}

\theoremstyle{plain}
\newtheorem{thm}{Theorem}
\newtheorem{lem}[thm]{Lemma}
\newtheorem{cor}[thm]{Corollary}
\newtheorem{prop}[thm]{Proposition}

\theoremstyle{definition}
\newtheorem{defn}{Definition}
\newtheorem{exa}{Example}

\usepackage{tikz}

\newcommand{\graphglyph}[1]{%
  \mathchoice%
    {\tikz[baseline=-0.5ex, line width=0.65pt, scale=1.0]{#1}}%
    {\tikz[baseline=-0.5ex, line width=0.65pt, scale=1.0]{#1}}%
    {\tikz[baseline=-0.4ex, line width=0.5pt,  scale=0.72]{#1}}%
    {\tikz[baseline=-0.35ex,line width=0.4pt,  scale=0.55]{#1}}%
}

\newcommand{\diamondgraph}{\graphglyph{%
  \coordinate (l) at (-0.65ex, 0);
  \coordinate (r) at ( 0.65ex, 0);
  \coordinate (t) at ( 0,      0.65ex);
  \coordinate (b) at ( 0,     -0.65ex);
  \draw (l) -- (t) -- (r) -- (b) -- (l) -- (r);
  \fill (l) circle (1pt) (r) circle (1pt) (t) circle (1pt) (b) circle (1pt);
}}

\newcommand{\bowtiegraph}{\graphglyph{%
  \coordinate (c)  at ( 0,      0);
  \coordinate (lt) at (-0.75ex, 0.55ex);
  \coordinate (lb) at (-0.75ex,-0.55ex);
  \coordinate (rt) at ( 0.75ex, 0.55ex);
  \coordinate (rb) at ( 0.75ex,-0.55ex);
  \draw (lt) -- (lb) -- (c) -- cycle;
  \draw (rt) -- (rb) -- (c) -- cycle;
  \fill (c) circle (1pt) (lt) circle (1pt) (lb) circle (1pt) (rt) circle (1pt) (rb) circle (1pt);
}}

\newcommand{\bullgraph}{\graphglyph{%
  \coordinate (lt) at (-0.45ex, 0.5ex);
  \coordinate (rt) at ( 0.45ex, 0.5ex);
  \coordinate (b)  at ( 0,     -0.35ex);
  \coordinate (ll) at (-0.95ex,-0.55ex);
  \coordinate (rl) at ( 0.95ex,-0.55ex);
  \draw (lt) -- (rt) -- (b) -- cycle;
  \draw (lt) -- (ll);
  \draw (rt) -- (rl);
  \fill (lt) circle (1pt) (rt) circle (1pt) (b) circle (1pt) (ll) circle (1pt) (rl) circle (1pt);
}}

\title{Chromatic numbers from edge 
ideals:\\ Graph classes with vanishing 
syzygies are polynomially 
$\chi$-bounded}
\author{ Alexander Engström\\ 
alexander.engstrom@fra.se }
\date\today

\begin{document}
\maketitle

\begin{abstract}
The chromatic number $\chi$ of a graph 
is bounded from below by its clique 
number $\omega,$ but the gap between these two values can be
arbitrary large. Perfect graphs are 
defined by the condition that
$\chi=\omega$ for all 
induced subgraphs. An interesting 
relaxation are $\chi$-bounded graph 
classes, where $\chi\leq f(\omega)$ for some suitable function $f.$
It  is not always possible to achieve this 
with a polynomial $f.$

The edge ideal $I_G$ of a graph $G$ is 
generated by monomials $x_ux_v$ for 
each edge $uv$ of $G.$ The bi-graded 
betti numbers $\beta_{i,j}(I)$ are 
central algebraic geometric invariants. 
We study the graph classes where for 
some fixed $i,j$ that syzygy vanishes, 
that is, $\beta_{i,j}(I_G)=0.$ 

We prove that $\chi\leq f(\omega),$ 
where $f$ is a polynomial of degree 
$2j-2i-4.$
For the elementary special case 
$\beta_{i,2i+2}(I_G)=0,$ this  amounts 
to that $(i+1)K_2$--free graphs are 
${\omega-1+2i \choose 2i}$--colorable, 
improving on an old combinatorial 
result by Wagon. 
We also show that triangle-free graphs 
with $\beta_{i,j}(I_G)=0$ are 
$(j-1)$--colorable.

Complexity wise, we show that these 
colorings can be derived in time 
$O(n^3)$ for graphs on $n$ vertices. 
Moreover, we show that for almost all 
graphs with parabolic $i,j,$ there are 
better bounds on $\chi.$
\end{abstract}

\section{Introduction}

For a graph $G,$ the largest number of 
vertices in a largest complete
subgraph is a natural lower bound for 
its chromatic number: $\omega(G) \leq 
\chi(G).$ In general that inequality 
may be arbitrary bad, as there are 
triangle-free graphs of any chromatic 
number. On the other hand, a perfect 
graph is characterized by that 
$\omega(H) = \chi(H)$ for all of its 
induced subgraphs $H.$ The last decades 
have seen an intense activity in 
extending our understanding of how to 
measure and categorize the perfectness 
of graphs. A class $\mathcal{F}$ of graphs is 
\emph{$\chi$--bounded} by the function 
$f$ if $\chi(G) \leq f(\omega(G))$ for 
all graphs $G \in \mathcal{F}.$ 
Note that the class of all graphs is not $\chi$-bounded: since there
exist triangle-free graphs with arbitrarily large chromatic number one could not define a value for $f(2)$ for any function $f$.
The 
central Gyárfás-Sumner conjecture, 
which has been open for fifty years, is 
that for any forest $F,$ the $F$--free 
graphs are $\chi$--bounded.
For established $\chi$--bounded graph 
classes the common next question is if 
the bounding function is polynomial or 
not. One may for any arbitrarily fast growing function $g$ construct a graph class  that it is optimally $\chi$-bounded by a function $f\geq g,$ as proved in \cite{REF1}. Examples of 
polynomially $\chi$--bounded graph 
classes are even-hole-free graphs 
\cite{CS23}, circle graphs \cite{Dav22,DM21} and bounded twin-width graphs 
\cite{BT25}.

For a graph $G$ on vertex set $V$ and 
edge set $E,$ its edge ideal in 
$\mathbf{k}[x_v \mid v\in V]$ is 
generated by the monomials $x_ux_v$ for 
$uv \in E(G).$
In discussing betti numbers 
$\beta_{i,j}(I_G)$ of edge ideals, we 
assume from now that $j\leq 2i+2,$ 
because otherwise they are trivially 
zero. For general combinatorial 
considerations regarding edge ideals, 
see for example \cite{de} or 
\cite{EO23}.
From a geometric stand point, the first 
interesting case is if almost all 
$\beta_{i,j}(I_G)$ vanish, except for 
if $j=i+2.$ Then the ideal is said to 
have a \emph{linear} resolution. In a 
seminal paper building a bridge between 
algebra and combinatorics, Fr\"oberg 
proved \cite{fro} that an edge ideal 
has a linear resolution if and only if 
the graph is the complement of a 
chordal graph. The complement of a 
chordal graph is perfect, setting the 
backdrop for this paper. 
In Fr\"oberg’s setting, we need an
infinite number of betti numbers to equal zero in order to obtain a combinatorial property
of the underlying graph. Here we study the implications of the vanishing of a single betti
number.
The last years 
a number of interesting combinatorial 
results have been established with this 
methodology, for example 
Randriambololona \cite{ran} won the 
best paper award of Eurocrypt 2025 
investigating the post-quantum 
cryptographic algorithm McEliece.

The main result of this paper is proved 
in Section 4: Corollary \ref{mainCor1} 
states that the class of graphs with $ 
\beta_{i,j}(I_G) = 0$ is 
$\chi$--bounded by
\[ {\omega-1+2(j-i-2) \choose 2(j-i-2)} 
+  {j-2 \choose 2(j-i-2)+1}, \] 
which is a polynomial in $\omega$ of 
degree $2(j-i-2).$ For triangle-free 
graphs the slightly stronger result  
$\chi(G) \leq j-1$ is proved in 
Corollary \ref{mainCor2}.

In Section 5 we show in Corollary 
\ref{asymBound} that if $(j-i)^2 \geq 
j+i+2$, then almost all graphs $G$ with 
$ \beta_{i,j}(I_G) = 0$ are 
$\chi$-bounded by
\[ {\omega-1+2(j-i-2) \choose 
2(j-i-2)}. \] 

In Section 6 we show in 
Corollary \ref{n3} that the graph 
colorings in our main result may be 
done in time $O(n^3)$ for graphs on $n$ 
vertices. For the triangle-free case, 
it is in time $O(n^2)$ according to 
Corollary \ref{n2}.

Finally, in Section 7, we discuss the pros and cons of the approach taken in this paper and suggest directions for future research.

Before reaching the main part of our 
paper, some elementary considerations 
on $\chi$-boundedness are recorded in 
Section 2. 
It is included because the proofs provide a conceptual foundation for the more technical versions of the same arguments in Section 3, geared toward applications in the algebraic setting.
Although it is all fairly 
elementary, it includes an improvement 
of an old influential result by Wagon 
\cite{W80}, proving that 
$(i+1)K_2$--free graphs are 
${\omega-1+2i \choose 2i}$--colorable.

\section{Ordinary polynomially 
$\chi$--bounded graph classes}

The chromatic number of a graph $G$ is 
denoted by $\chi(G).$ The order of the 
largest clique of $G$ is denoted by 
$\omega(G).$

\begin{defn} A hereditary class of 
graphs is \emph{$\chi$-bounded} by a function $f$ 
if $\chi(G)\leq f(\omega(G))$ for all 
graphs $G$ in the class.
\end{defn}

The disjoint union of two graphs $G$ 
and $H$ is denoted $G \cup H$ and 
multiples by $pG=\cup_p G.$
The join $G+H$ is $G \cup H$ with 
all edges $uv$ where $u \in V(G)$ and $v \in V(H).$
For a 
graph $G$ and a set $\bf{H}$ of graphs, 
let $ G \cup {\bf H  } = \{ G \cup H 
\mid H \in {\bf H} \}$ and  $ G + {\bf 
H  } = \{ G + H \mid H \in {\bf H} \}.$ 
The induced subgraph on $U\subseteq 
V(G)$ of $G$ is denoted $G[U].$ Basic 
graphs on $n$ vertices are the complete 
graph $K_n,$ the cycle $C_n$ and the 
path $P_n.$ For a graph $H$, the class of $H$--free graphs are those without $H$ as an induced subgraph. For a set ${\bf H}$ of graphs, the class of ${\bf H}$--free graphs are those without any $H \in {\bf H}$ as an induced subgraph.

In the remainder of this section we will many times consider how the classes of ${\bf H}$--free graphs and $K_2 \cup {\bf H  } $--free graphs are related. Although matchings are of intrinsic interest to graph theory, the main reason here is that if $\beta_{i,j}(I_G)>0$ then $\beta_{i+1,j+2}(I_{K_2 \cup G})>0,$ and that is useful in transferring results between betti numbers on different indices. 

Wagon \cite{W80} introduced the following $f_n^W(\omega)$ functions 
and proved this influential result:
\begin{prop}\label{wagon} If 
$f_1^{\mathrm{W}}(\omega)=1$ and 
$f_{n+1}^{\mathrm{W}}(\omega)={ \omega 
\choose 2} f_n^{\mathrm{W}}(\omega) + 
\omega,$ then the class of $pK_2$--free 
graphs is $\chi$--bounded by 
$f_p^{\mathrm{W}}(\omega).$ In 
particular, the class of $2K_2$--free 
graphs is $\chi$--bounded by ${\omega 
+1 \choose 2}.$
\end{prop}
The proof of Lemma \ref{K2extend} below 
builds on the proof of Proposition 
\ref{wagon} by Wagon, who attributes 
the basic idea to Erdös and Hajnal. 
The main difference between
Wagon’s argument and our proof of Lemma~\ref{K2extend} is
that the chromatic number for some 
induced subgraphs $G[C_{i,j}]$ are 
better estimated using an idea by 
Bharati and Chodum \cite{BC18}. But the 
main point of Lemma \ref{K2extend} is 
actually that it serves as a model for 
the more technical proof of Theorem 
\ref{bigthm}, which is crucial to 
derive the main results of this paper.

\begin{lem}\label{K2extend}
If ${\bf H}$ is a set of graphs and the 
class of ${\bf H}$--free graphs is  
$\chi$-bounded by $f_{\bf H},$ then the 
class of  $K_2 \cup {\bf H  } $--free 
graphs is $\chi$-bounded by
\[
f_{K_2 \cup {\bf H  } }(\omega) = 
\sum_{k=1}^{\omega}  (\omega-k+1)  
f_{{\bf H  } }(k).
\]
\end{lem}

\begin{proof}
Let $G$ be a $K_2 \cup {\bf H  } 
$--free graph with a clique on the 
vertices $1,2,\ldots, \omega=\omega(G).$
For $1\leq i < j \leq \omega$ define
\[
C_{i,j}' = \{ v\in V(G) \setminus 
\{1,2,\ldots, \omega\} \mid vi,vj 
\not\in E(G)  \}.
\]
For any induced subgraph $H$ of 
$C_{i,j}'$, the graph $K_2 \cup H$ on 
$\{i,j\} \cup V(H)$ is an induced 
subgraph of $G.$ By assumption, $G$ is 
a $K_2 \cup {\bf H  } $--free graph, 
and, thus, the induced subgraph of $G$ 
on $C_{i,j}' $ is an  ${\bf H  } 
$--free graph.

Set
\[
C_{i,j} = C_{i,j}' \setminus  
\left(\,\, \bigcup_{k=1}^{i-1} 
\bigcup_{l=k+1}^\omega C_{k,l}'   \quad 
\cup \quad  \bigcup_{l=i+1}^{j-1} 
C_{i,l}'  \,\,  \right). 
\]
If $v\in C_{i,j}$ and $1\leq k < i$ 
then $kv \in E(G)$ as 
\[ v\in C_{i,j}' \Rightarrow iv \not\in 
E(G) \textrm{ and }  jv \not\in E(G) \]
and
\[ v\not\in C_{k,j}' \Rightarrow kv \in 
E(G) \textrm{ or }  jv \in E(G). \]
If $v\in C_{i,j}$ and $i< l < j$ then 
$lv \in E(G)$ as 
\[ v\in C_{i,j}' \Rightarrow iv \not\in 
E(G) \textrm{ and }  jv \not\in E(G) \]
and
\[ v\not\in C_{i,l}' \Rightarrow iv \in 
E(G) \textrm{ or }  lv \in E(G). \]
It follows that any clique in 
$G[C_{i,j}]$ extends to a larger clique 
in $G$ by the $j-2$ vertices
$\{1,2,\ldots,i-1\} \cup \{ 
i+1,i+2,\ldots, j-1 \}.$ 
Thus
\[ \omega(G[C_{i,j}]) \leq 
\omega(G)-j+2\]
and
\[ \chi(G[C_{i,j}]) \leq f_{{\bf H  } 
}( \omega(G)-j+2 ). \]

So far we have partitioned the vertices 
in $G$ outside the clique on 
$1,2,\ldots, \omega=\omega(G)$  that 
has at least two non-edges to the 
clique. 
A vertex outside the clique cannot be adjacent to all of its vertices, as the clique is maximal.
For $1\leq i \leq \omega(G)$ set
\[ D_i'
= \left\{ v\in V(G) \setminus 
\{1,2,\ldots, \omega\} \mid vi \not\in 
E(G); \{v1,v2, \ldots, v\omega \} 
\setminus vi \subseteq E(G) 
 \right \}
\]
If there would be an edge between two 
vertices $u$ and $v$ in $D_i',$ then 
the vertices
\[ 1,2, \ldots, i-1, i+1, \ldots, 
\omega, u, v \]
would constitute a clique on $\omega+1$ 
vertices, which is a contradiction. 
Thus, $D_i'$ is independent. As there 
are no edges from $i$ to vertices of 
$D_i',$ the set
\[ D_i = \{i\} \cup D_i' \]
is independent for all $i.$ We only 
need one color for an independent set, 
and
\[ \chi(G[D_i]) = 1 \leq  f_{{\bf H  } 
}( 1). \]
Coloring each part of the partition
\[ V(G) \quad = \quad   
\bigcup_{i=1}^\omega D_{i}      \quad 
\cup \quad    \bigcup_{1\leq i < j \leq 
\omega}  C_{i,j}  \]
gives
\[\begin{array}{rcl} 
\chi(G) & \leq & \displaystyle  
\sum_{i=1} ^\omega \chi(G[D_i]) +  
\sum_{1\leq i < j \leq \omega}  
\chi(G[C_{i,j}]) \\
& \leq & \displaystyle  \sum_{i=1} 
^\omega  f_{{\bf H  } }( 1) +  
\sum_{1\leq i < j \leq \omega}   
f_{{\bf H  } }( \omega-j+2 ) \\
& = & \displaystyle  \omega  f_{{\bf H  
} }( 1) +  \sum_{j=2}^\omega   
(j-1)f_{{\bf H  } }( \omega-j+2 ) \\
& = & \displaystyle  
\sum_{j=2}^{\omega+1}   (j-1)f_{{\bf H  
} }( \omega-j+2 ) \\
& = & \displaystyle \sum_{k=1}^{\omega} 
 (\omega-k+1)  f_{{\bf H  } }(k).
\end{array}
\]
\end{proof}

The following observation is completely elementary. 
\begin{lem}\label{doubleBumpLemma}
For non-negative integers $n$ and $m,$
\[ 
\sum_{i=0}^n   (n+1-i)    {i + m   
\choose m } =  {n + m+2   \choose m+2}.
\]
\end{lem}
\begin{proof} The right hand side 
counts the number of order $m+2$ 
subsets of $\{0,1,\ldots,n+m+1\}.$ Put 
those subsets in buckets depending on 
the value of the second largest element 
of each subset. The value of the second 
largest element is $m+i$ for $0\leq i 
\leq n.$ There are $ {i + m   \choose m 
} $ ways to select the $m$ smallest 
elements and $ (n+1-i) $ ways to select 
the largest element. Summing up the 
sizes of the buckets gives the equality.
\end{proof}

\begin{thm}\label{pK2extend}
If ${\bf H}$ is a set of graphs and the 
class of ${\bf H}$--free graphs is  
$\chi$-bounded by $f_{\bf H},$ then the 
class of $pK_2 \cup {\bf H  } $--free 
graphs is $\chi$-bounded by
\[
f_{pK_2 \cup {\bf H  } }(\omega) = 
\sum_{k=1}^{\omega}   {\omega-k + 2p-1  
\choose 2p-1}    f_{{\bf H  } }(k).
\]
\end{thm}

\begin{proof}
For $p=1$ this is exactly Lemma 
\ref{K2extend}. Now we proceed by 
induction on $p>1.$ We assume that 
\[
f_{(p-1)K_2 \cup {\bf H  } }(\omega) = 
\sum_{k=1}^{\omega}   {\omega-k + 
2(p-1)-1  \choose 2(p-1)-1}    f_{{\bf 
H  } }(k) =  \sum_{k=1}^{\omega}   
{\omega-k + 2p-3  \choose 2p-3}    
f_{{\bf H  } }(k) 
\]
is $\chi$-bounding the class of 
$(p-1)K_2 \cup {\bf H  }$--free graphs.
According to Lemma \ref{K2extend},
\[
\begin{array}{rcl}
f_{pK_2 \cup {\bf H  } }(\omega) &= & 
\displaystyle \sum_{k=1}^{\omega}  
(\omega-k+1)  f_{ (p-1)K_2 \cup  {\bf H 
 } }(k) \\
&= & \displaystyle \sum_{k=1}^{\omega}  
(\omega-k+1)   \sum_{l=1}^{k}   {k-l + 
2p-3  \choose 2p-3}    f_{{\bf H  } 
}(l) \\
&= & \displaystyle \sum_{k=1}^{\omega} 
\sum_{l=1}^{k}   (\omega-k+1)    {k-l + 
2p-3  \choose 2p-3}    f_{{\bf H  } 
}(l) \\
&= & \displaystyle \sum_{l=1}^{\omega} 
\sum_{k=l}^{\omega}   (\omega-k+1)    
{k-l + 2p-3  \choose 2p-3}    f_{{\bf H 
 } }(l) \\
\end{array}
\]
is $\chi$-bounding the class of $pK_2 
\cup {\bf H  }$--free graphs. According 
to Lemma \ref{doubleBumpLemma} with  
$m=2p-3, $ and $i=k-l$ and $n=\omega-l,$
\[ 
\sum_{k=l}^{\omega}   (\omega-k+1)    
{k-l + 2p-3  \choose 2p-3} =  {\omega-l 
+ 2p-1  \choose 2p-1}
\]
and
\[
\begin{array}{rcl}
f_{pK_2 \cup {\bf H  } }(\omega) &= &  
\displaystyle \sum_{l=1}^{\omega}  
{\omega-l + 2p-1  \choose 2p-1}    
f_{{\bf H  } }(l). \\
\end{array}
\]
\end{proof}

Here is another elementary lemma.
\begin{lem}\label{singleBumpLemma}
$ \sum_{i=0}^n {m+i \choose m} = { 
n+m+1 \choose m+1 }. $
\end{lem}
\begin{proof}
Consider selection of $m+1$ elements of 
$\{0,1,\ldots, n+m\}$ and bucket 
depending on the largest element and 
sum up each bucket. 
\end{proof}

This will be useful in the algebraic 
setting:

\begin{cor}\label{corConstF}
 If $f_{{\bf H  } }(1) = 1$ and   
$f_{{\bf H  } }(\omega) = c_{\bf H  }$ 
for $\omega > 1$ is a $\chi$--bounding 
function for ${\bf H  }$--free graphs, 
then
\[ f_{pK_2 \cup {\bf H  } }(\omega) =  
{\omega + 2p-2  \choose 2p-1}  +   
c_{\bf H  }  {\omega + 2p-2  \choose 
2p} \]
is a $\chi$--bounding function for ${ 
pK_2 \cup \bf H  }$--free graphs.
\end{cor}
\begin{proof}
According to Theorem \ref{pK2extend} 
and Lemma \ref{singleBumpLemma},
\[ \begin{array}{rcl}
f_{pK_2 \cup {\bf H  } }(\omega) & = & 
\displaystyle  {\omega-1 + 2p-1  
\choose 2p-1}  +   c_{\bf H  }   
\sum_{k=2}^{\omega}   {\omega-k+ 2p-1  
\choose 2p-1}  \\
& = & \displaystyle  {\omega+ 2p-2  
\choose 2p-1}  +   c_{\bf H  }   
\sum_{i=0}^{\omega-2}   {i + 2p-1  
\choose 2p-1}  \\
& = & \displaystyle  {\omega + 2p-2  
\choose 2p-1}  +   c_{\bf H  }  {\omega 
+ 2p-2  \choose 2p}.     \\
\end{array}
\]
\end{proof}

\begin{cor}\label{corBoundingSimple}
If ${\bf H}$ is the set of graphs on 
$n$ vertices of maximal degree $n-1,$ 
then 
\[ f_{pK_2 \cup {\bf H  } }(\omega) =  
{\omega + 2p-2  \choose 2p-1}  +   
(n-1)  {\omega + 2p-2  \choose 2p} \]
is a $\chi$--bounding function for ${ 
pK_2 \cup \bf H  }$--free graphs.
\end{cor}
\begin{proof}
The ${\bf H}$-free graphs are of 
maximal degree at most $n-2,$ and 
$(n-1)$--colorable. Use Corollary 
\ref{corConstF}.
\end{proof}

\begin{cor}\label{corInducedMatching} 
For $pK_2$--free graphs,
\[
f_{pK_2}(\omega) =  {\omega-1 + 2(p-1)  
\choose 2(p-1)}
\]
is a $\chi$--bounding function.
\end{cor}
\begin{proof} Use ${\bf H}=\{K_2\}$ and 
$ c_{\bf H  }=1$ in Corollary 
\ref{corConstF} with $(p-1)K_2$ to get
\[ f_{pK_2}(\omega) = {\omega + 
2(p-1)-2  \choose 2(p-1)-1}  +   1  
{\omega + 2(p-1)-2  \choose 2(p-1)} = 
{\omega-1 + 2(p-1)  \choose 2(p-1)}
\]
\end{proof}

Let's compare Corollary 
\ref{corInducedMatching} to Wagon's old 
result in Proposition \ref{wagon}. It 
is easy to prove that:
\begin{prop}
\[ \lim_{\omega \rightarrow \infty} 
\,\,   
\frac{f_p^{\mathrm{W}}(\omega)}{\omega^{
2(p-1)}} = \frac{1}{2^{p-1}} 
\quad
\textrm{and}
\quad
\lim_{\omega \rightarrow \infty} \,\,   
\frac{f_{pK_2}(\omega)}{\omega^{2(p-1)}}
 = \frac{1}{(2p-2)!} 
\]
\end{prop}
So, for large $\omega,$ the new bound 
improves on Wagon's by a factor of 
$\frac{(2p-2)!}{2^{p-1}}.$
\begin{prop}
If $\omega>2$ and $p>2$ then
\[ f_p^{\mathrm{W}}(\omega) > 
f_{pK_2}(\omega) \]
and otherwise they are equal.
\end{prop}
\begin{proof} The recursion to prove 
Wagon's $f_p^{\mathrm{W}}(\omega)$ is a 
simpler version of the proof of Lemma 
\ref{K2extend}. Instead of the 
inequality
\[ \chi(G[C_{i,j}]) \leq 
f_{p-1}^{\mathrm{W}} ( \omega(G)-j+2 ) 
\]
the employed 
\[ \chi(G[C_{i,j}]) \leq 
f_{p-1}^{\mathrm{W}} ( \omega(G) ) \]
is worse when $\omega>2$ and $p>2$. 
That they are equal otherwise is 
straightforward.
\end{proof}

Here is another corollary of 
Theorem~\ref{pK2extend}.
\begin{cor} If $\bf H$ is a set of 
graphs such that all $\bf H$-free 
graphs are perfect, then
\[ f_{pK_2 \cup {\bf H  } }(\omega) = { 
\omega + 2p \choose 2p+1 } \]
is a $\chi$--bounding function for ${ 
pK_2 \cup \bf H  }$--free graphs.
\end{cor}
\begin{proof}
The $\bf H$-free graphs are 
$\chi$-bounded by $ f_{{\bf H  } }(k) = 
k$ as they are perfect. By 
Theorem~\ref{pK2extend} and 
Lemma~\ref{doubleBumpLemma},
\[
\begin{array}{rcl}
\displaystyle
f_{pK_2 \cup {\bf H  } }(\omega) & =& 
\displaystyle  \sum_{k=1}^{\omega}   
{\omega-k + 2p-1  \choose 2p-1}    
f_{{\bf H  } }(k)  \\ 
& = &  \displaystyle  
\sum_{k=1}^{\omega}   {\omega-k + 2p-1  
\choose 2p-1}  k \\
& = & \displaystyle  
\sum_{i=0}^{\omega-1}  ((\omega-1)+1-i) 
  {i + 2p-1  \choose 2p-1}  \\
  & = &  \displaystyle { \omega - 1 + 
2p -1 + 2\choose 2p-1 +2 } \\
  & = &  \displaystyle  { \omega + 2p 
\choose 2p+1 }
\end{array}
\]
is a $\chi$--bounding function for ${ 
pK_2 \cup \bf H  }$--free graphs.
\end{proof}

The following two results were proved 
by  Bharathi and Choudum \cite{BC18} in 
the case $p=1.$
\begin{cor}
The ${ pK_2 \cup P_4}$--free graphs are 
$\chi$-bounded by $\omega + 2p \choose 
2p+1$.
\end{cor}
\begin{proof}
$P_4$--free graphs are perfect 
\cite{S74}.
\end{proof}
\begin{cor}
The ${ pK_2 \cup P_3}$--free graphs are 
$\chi$-bounded by $\omega + 2p \choose 
2p+1$.
\end{cor}
\begin{proof}
$P_3$--free graph are disjoint unions 
of cliques, and thus perfect.
\end{proof}

For graphs that are almost perfect, 
similar results are possible:

\begin{cor}
The $\{ (p+2)K_2, pK_2 \cup 
\diamondgraph \}$--free graphs are 
$\chi$-bounded by $ {\omega + 2p 
\choose 2p+1} +   {\omega + 2p-3  
\choose 2p-1}  $.
\end{cor}
\begin{proof}
Let ${\bf H} = \{2K_2, \diamondgraph \}.$
According to Theorem 4 of \cite{BC18}, 
the $\bf H$-free graphs are 
$\chi$-bounded by $ f_{{\bf H  } }(k) = 
k$ for all $k,$ with the exception $ 
f_{{\bf H  } }(2) = 3.$
By Theorem~\ref{pK2extend} and 
Lemma~\ref{doubleBumpLemma}, 
\[
\begin{array}{rcl}
\displaystyle
f_{pK_2 \cup {\bf H  } }(\omega) & =& 
\displaystyle  \sum_{k=1}^{\omega}   
{\omega-k + 2p-1  \choose 2p-1}    
f_{{\bf H  } }(k)  \\ 
& = &  \displaystyle   {\omega-2 + 2p-1 
 \choose 2p-1}  + \sum_{k=1}^{\omega}   
{\omega-k + 2p-1  \choose 2p-1}  k \\
& = &  \displaystyle   {\omega-2 + 2p-1 
 \choose 2p-1}  + { \omega + 2p \choose 
2p+1 }.
\end{array}
\]
\end{proof}

A generalization of perfect graphs are the perfectly divisible graphs.
A graph 
$G$ is said to be \emph{perfectly 
divisible} if for all induced subgraphs 
$H$ of $G$, $V(H)$ can be partitioned 
into two sets $A$ and $B$ such that 
$H[A]$ is perfect and 
$\omega(H[B])<\omega(H).$ This lemma 
was proved by Chudnovsky and Sivaraman 
\cite{CS19}.
\begin{lem}
If $G$ is a perfectly divisible graph, 
then $\chi(G) \leq {\omega(G)+1 \choose 
2}$.
\end{lem}
Along the lines above, it is 
straightforward to prove this:
\begin{cor} If ${\bf H}$ is a set of 
graphs such that all ${\bf H}$--free 
graphs are perfectly divisible, then 
the $(p-1)K_2\cup {\bf H}$--free graphs 
are $\chi$--bounded by ${\omega-1 + 2p  
\choose 2p}.$
\end{cor}
For example, Chudnovsky and Sivaraman 
\cite{CS19} proved that 
$\{P_5,\bullgraph \}$--free graphs are 
perfectly divisible.
\begin{cor} The $(p-1)K_2\cup  
\{P_5,\bullgraph \}$--free graphs are 
$\chi$--bounded by ${\omega-1 + 2p  
\choose 2p}.$
\end{cor}

\section{A class of graphs for 
vanishing syzygies}

This section is completely graph theoretic, we introduce the sets ${\bf B}_{n,d}$ of graphs and prove properties of ${\bf B}_{n,d}$--free graphs. The only reason to introduce them is in their relation to the algebraic results in the next section, where we explain how the ${\bf B}_{n,d}$--free graphs are related to the vanishing of syzygies.

\begin{defn}
Let ${\bf B}_{n,0}$ be the set of 
graphs on $n\geq 2$ vertices with maximal 
degree $n-1$, and let 
\[{\bf B}_{n,d} = \bigcup_{m=2d}^{n-2} 
K_{n-m-2}+(K_2\cup {\bf B}_{m,d-1}) \] 
for any integers $d>0$ and $n\geq 
2(d+1).$
\end{defn}

\begin{exa}
The bowtie $\bowtiegraph$ is in $B_{5,1}.$
\end{exa}

\begin{defn}
For any integers $d\geq 0$ and $n\geq 
2(d+1)$ let the function $g_{n,d}: 
\{1,2,\ldots\} \rightarrow 
\{1,2,\ldots\}$ be
\[\begin{array}{rcll}
g_{n,d}(1) &= &1 & \textrm{for $d\geq 
0$ and $n\geq 2(d+1),$} \\
g_{n,0}(\omega) &= & n-1 & \textrm{for 
$\omega > 1$ and $n\geq 2,$}\\
g_{n,d}(\omega) & = & \displaystyle  
\sum_{k=1}^{\omega} (\omega-k+1) 
g_{\max(n+k-\omega-2,2d),d-1}( k ) &  
\textrm{for $\omega > 1, d>0,$ and 
$n\geq 2(d+1).$}\\
\end{array}
\]
\end{defn}

\begin{thm}\label{bigthm}
For any integers $d\geq 0$ and $n\geq 
2(d+1)$ the class of ${\bf 
B}_{n,d}$--free graphs is 
$\chi$--bounded by $g_{n,d}(\omega).$
\end{thm}

\begin{proof}
We are to show that $\chi(G) \leq 
g_{n,d}(\omega(G))$ for every ${\bf 
B}_{n,d}$--free graph $G.$

First note that for graphs with 
$\omega(G)=1$ it is satisfied, as 
$g_{n,d}(1)=1$ for all $n,d.$ In 
calculations later on in the proof, 
$g_{n,d}(1)$ will sometimes be replaced 
by $g_{\ast,\ast}(1),$ as the value of 
the function at $1$ is independent of 
the index $n,d.$

The proof will be by induction on $d.$ 
The base case is $d=0.$ By definition, 
${\bf B}_{n,0}$ is the set of graphs on 
$n$ vertices of maximal degree $n-1.$ A 
${\bf B}_{n,0}$--free graph $G$ 
contains no induced subgraph on 
$n$ vertices of maximal degree $n-1,$ 
or equivalently, the maximal degree of 
$G$ is at most $n-2.$ As the maximal 
degree of $G$ is at most $n-2,$ one may 
color it with $n-1$ colors, and 
$\chi(G) \leq n-1 = g_{n,0}(\chi(G))$ 
as desired.

Now to the induction step. We are 
provided with a ${\bf B}_{n,d}$--free 
graph $G$ and $d>0.$ The induction step 
follows the setup of the proof of Lemma 
\ref{K2extend}. Find a clique on the 
vertices $1,2,\ldots, 
\omega=\omega(G),$ and define the sets 
$C_{i,j}$ and $D_i$ as in that proof. 
For any induced subgraph $H$ of $G$ on 
$U\subseteq C_{i,j},$ the induced 
subgraph of $G$ on $U \cup 
\{1,2,\ldots,j\}$ is isomorphic to 
$K_{j-2}+(K_2 \cup H).$

If $n-j \geq 2d$ then any induced 
subgraph $H$ of $C_{i,j}$ on $n-j$ 
vertices turns into the induced 
subgraph $K_{j-2}+(K_2 \cup H)$ of $G$ 
on $n$ vertices. By assumption, the 
graph $G$ is ${\bf B}_{n,d}$--free, 
where
$ {\bf B}_{n,d} \supseteq 
K_{j-2}+(K_2\cup {\bf B}_{n-j,d-1}).$ 
This shows that $G[C_{i,j}]$ is  ${\bf 
B}_{n-j,d-1}$--free.

If $n-j < 2d$ then any induced subgraph 
$H$ of $C_{i,j}$ on $2d$ vertices turns 
into the induced subgraph $K_{j-2}+(K_2 
\cup H)$ of $G$ on $j+2d>n$ vertices. 
Pick any subgraph $K_{n-2d-2}$ of 
$K_{j-2}$ to get an induced subgraph 
$K_{n-2d-2}+(K_2 \cup H)$ of $G.$  By 
assumption, the graph $G$ is ${\bf 
B}_{n,d}$--free, where $ {\bf B}_{n,d} 
\supseteq K_{n-2d-2}+(K_2\cup {\bf 
B}_{2d,d-1}).$ This shows that 
$G[C_{i,j}]$ is  ${\bf 
B}_{2d,d-1}$--free.

Now proceed as in the Proof of Lemma 
\ref{K2extend} with $\chi(G[D_i])\leq  
1 =  g_{\ast,\ast}(1)$ 
to evaluate the 
sum:
\[\begin{array}{rcl} 
\chi(G) & \leq & \displaystyle  
\sum_{i=1} ^\omega \chi(G[D_i]) +  
\sum_{1\leq i < j \leq \omega}  
\chi(G[C_{i,j}]) \\
& \leq & \displaystyle  \sum_{i=1} 
^\omega  g_\ast ( 1) +  \sum_{1\leq i < 
j \leq \omega}   g_{\max(n-j,2d),d-1}( 
\omega-j+2 ) \\
& = & \displaystyle  \sum_{1\leq i < j 
\leq \omega+1}   g_{\max(n-j,2d),d-1}( 
\omega-j+2 ) \\
& = & \displaystyle  
\sum_{j=2}^{\omega+1} (j-1) 
g_{\max(n-j,2d),d-1}( \omega-j+2 ) \\
& = & \displaystyle  
\sum_{k=1}^{\omega} (\omega-k+1) 
g_{\max(n+k-\omega-2,2d),d-1}( k ) \\
& = & g_{n,d}(\omega).
\end{array}
\]
\end{proof}

\begin{prop}\label{prepp} For $\omega 
\geq 1$ it holds that
\[ g_{n,d}(\omega) \leq {\omega-1+2d 
\choose 2d} +  {n-2 \choose 2d+1} \]
and it is sharp exactly when $(d=0, 
\omega =1, n>2)$ or $(d>0, 
n>\omega+2d+1).$
\end{prop}
\begin{proof} First we verify the case 
$d=0,$
\[
{\omega-1+2d \choose 2d} +  {n-2 
\choose 2d+1} = {\omega-1 \choose 0} +  
{n-2 \choose 1} = 1 + n -2= n-1.
\]
By definition, $g_{n,0}(\omega)=n-1$ 
for $\omega>1,$ so the proposition 
statement is an equality in those 
cases. 
By definition, $g_{n,0}(\omega)=1$ for 
$\omega=1,$ so the proposition 
statement is an equality for $n=2$ and 
a strict inequality for $n>2.$

Now we continue by induction on $d>0.$ 
Consider that 
\[  \begin{array}{rcl}
g_{n,d}(\omega) & = &  \displaystyle 
\sum_{k=1}^{\omega} (\omega-k+1) 
g_{\max(n+k-\omega-2,2d),d-1}( k )  \\
& \leq & \displaystyle 
\sum_{k=1}^{\omega} (\omega-k+1) \left( 
   {k-1+2(d-1) \choose 2(d-1)} +  
{\max(n+k-\omega-2,2d) -2 \choose 
2(d-1)+1}     \right). \\
\end{array} \]
Let's break down the two parts of the 
sum. The first part is 
\[
\sum_{k=1}^{\omega} (\omega-k+1)   
{k-1+2(d-1) \choose 2(d-1)}  = 
\sum_{i=0}^{\omega-1} ((\omega-1)+1-i)  
 {i+2d-2 \choose 2d-2}  = { \omega - 1 
+2d \choose 2d }
\]
by Lemma \ref{doubleBumpLemma}. The 
second part is 
\[
\sum_{k=1}^{\omega} (\omega-k+1)  
{\max(n+k-\omega-2,2d) -2 \choose 
2(d-1)+1}.
\]
The binomial coefficient vanishes 
unless $n+k-\omega-2 \geq 2d+1,$ or 
equivalently, that $k \geq \omega - n + 
2d + 3.$ The second part may be 
rewritten as:
\[
\begin{array}{rcl}
\displaystyle  \sum_{k=\omega - n + 2d 
+ 3}^{\omega} (\omega-k+1)  
{n+k-\omega-4 \choose 2d-1} &  = & 
\displaystyle  \sum_{i=0}^{n - 2d - 3} 
((n - 2d - 3)+1-i)  {i + 2d -1 \choose 
2d-1}    \\
    & = & \displaystyle  {n - 2d - 3 + 
2d -1 +2  \choose 2d-1 + 2 }    \\
    & = & \displaystyle  {n -2   
\choose 2d+1 },    \\
\end{array}
\]
with an application of Lemma 
\ref{doubleBumpLemma}. Adding the two 
pieces up proves the inequality.

Now to the question of when the 
inequality is an equality or sharp. For 
$d=0$ it's sharp exactly when 
$\omega=1$ and $n>2.$ For $d=1,$ in the 
expansion of $g_{n,1}(\omega),$ the 
terms are of the type:
\[ g_{\max(n+k-\omega-2,2d),d-1}( k ) \]
for $1 \leq k \leq \omega .$ The only 
one that could be sharp is for $k=1:$
\[ g_{\max(n-\omega-1,2),0}( 1 ) \]
The condition is that $n-\omega-1>2,$ 
or equivalently,  $n>\omega+3 $ for it 
to be sharp. 

For $d>1$ we prove it by induction on 
$d.$ The inequality for 
$g_{n,d}(\omega)$ is sharp exactly when 
$n> \omega +2d+1.$ We expanded 
$g_{n,d}(\omega)$ into terms of the type
\[ g_{\max(n+k-\omega-2,2d),d-1}( k ). 
\]
By induction, the inequality for one of 
them is sharp exactly when
\[
\max(n+k-\omega-2,2d) > k + 2(d-1) + 1.
\]
It can never be that $2d> k + 2(d-1) + 
1$, as $k\geq 1.$ It remains that
\[ n+k-\omega-2>2d \,\, \textrm{ and } 
\,\, n+k-\omega-2> k + 2(d-1) + 1. \]
The rightmost inequality implies the 
leftmost one, as $k\geq 1.$ It 
simplifies to
the desired
\[ n> \omega +2d+1
\]
and we are done.
\end{proof}

This proposition will be useful in considering triangle free graphs, for which slightly better bounds than for general graphs will be derived.

 \begin{prop}\label{gTriangleFree} $ 
g_{n,d}(2) = n-1$ for $d\geq 0$ and 
$n\geq 2(d+1).$
\end{prop}
\begin{proof} We prove the statement by induction 
on $d.$ For $d=0$ it is true by 
definition. Assume that $d>0.$
By the assumption $n\geq 2(d+1),$
\[ \max(n-2,2d) = n-2. \] 
By the definition of $g_{n,d}(\omega)$ and the induction assumption,
\[
\begin{array}{rcl}
g_{n,d}(2)  
& = &  (2-1+1)  g_{\max(n+1-2-2,2d),d-1}( 1 ) +   (2-2+1)  g_{\max(n+2-2-2,2d),d-1}( 2 )  \\
& = &  2 \cdot 1+   g_{\max(n-2,2d),d-1}( 2 )  \\
& = &  2  +   g_{n-2,d-1}( 2 )  \\
& = &  2 + n -3    \\
& = &  n -1.     \\
\end{array}
\]

\end{proof}

 We record the exact values of 
$g_{n,d}(\omega)$ for some small $n,d.$ 
The proposition above applies exactly 
for $d>0$ when $\omega \geq n - 2d - 1 
.$
 
\begin{exa}  For $\omega > 1$ the 
functions 
 \[\begin{array}{|c|c|c|c|c|c|}
 \hline
 g_{2,0}(\omega) & g_{3,0}(\omega) & 
g_{4,0}(\omega) & g_{5,0}(\omega) & 
g_{6,0}(\omega) & g_{7,0}(\omega) \\
 \hline & g_{4,1}(\omega)& 
g_{5,1}(\omega)  & g_{6,1}(\omega)  & 
g_{7,1}(\omega) & g_{8,1}(\omega)  \\
  \hline && g_{6,2}(\omega)& 
g_{7,2}(\omega)  & g_{8,2}(\omega)  & 
g_{9,2}(\omega) \\
   \hline &&& g_{8,3}(\omega)  & 
g_{9,3}(\omega)  & g_{10,3}(\omega) \\
    \hline &&&& g_{10,4}(\omega)  & 
g_{11,4}(\omega) \\
     \hline &&&&& g_{12,5}(\omega) \\
 \hline
 \end{array} \]
 are
 \[\begin{array}{|c|c|c|c|c|c|}
 \hline
 \omega \mapsto 1 &  \omega \mapsto 2 & 
 \omega \mapsto 3& \omega \mapsto 4 &  
\omega \mapsto 5 &  \omega \mapsto 6  \\
 \hline &  \omega \mapsto {\omega+1 
\choose 2}  &  \omega \mapsto {\omega+1 
\choose 2}+1 &  2 \mapsto 5  & 2 
\mapsto 6 & 2 \mapsto 7 \\
 &   &  &  \omega \mapsto {\omega+1 
\choose 2}+4  & 3 \mapsto 13& 3 \mapsto 
16  \\
 &   &  &  \textrm{for $\omega \geq 3$} 
 &  \omega \mapsto {\omega+1 \choose 
2}+10 & 4 \mapsto 26  \\
  &   &  & &   \textrm{for $\omega \geq 
4$}  &  \omega \mapsto {\omega+1 
\choose 2}+20  \\
   &   &  & & &   \textrm{for $\omega 
\geq 5 $}    \\
  \hline && \omega \mapsto {\omega+3 
\choose 4} & \omega \mapsto {\omega+3 
\choose 4}+1  & 2 \mapsto 7 & 2 \mapsto 
8 \\
    &&&  & \omega \mapsto {\omega+3 
\choose 4}+6  & 3 \mapsto 26  \\
     &&&  & \textrm{for $\omega \geq 3 
$} &  \omega \mapsto {\omega+3 \choose 
4}+21    \\
      &&&&  & \textrm{for $\omega \geq 
4 $}   \\
   \hline &&&  \omega \mapsto {\omega+5 
\choose 6}  & \omega \mapsto {\omega+5 
\choose 6} +1   & 2\rightarrow 9   \\
   &&&   &  &  \omega \mapsto {\omega+5 
\choose 6} + 8 \\
    &&&&  & \textrm{for $\omega \geq 3 
$}   \\
    \hline &&&& \omega \mapsto 
{\omega+7 \choose 8}    & \omega 
\mapsto {\omega+7 \choose 8} +1 \\
     \hline &&&&&  \omega \mapsto 
{\omega+9 \choose 10}  \\
 \hline
 \end{array} \]
 where the only values not given by the 
results above are
 \[ g_{7,1}(3) =  3 g_{3,0}(1) + 2 
g_{4,0}(2) +  g_{5,0}(3)  =  3\cdot 1 + 
2 \cdot 3 + 4 = 3+6+4 = 13, \]
 \[ g_{8,1}(3) =  3 g_{4,0}(1) + 2 
g_{5,0}(2) +  g_{6,0}(3)  = 3\cdot 1 + 
2 \cdot 4 + 5 =  16, \]
 \[ g_{8,1}(4) =  4 g_{5,0}(1)  + 3 
g_{4,0}(2) + 2 g_{5,0}(3) +  g_{6,0}(4) 
= 4\cdot 1 + 3 \cdot 3 + 2\cdot 4 + 5 = 
26, \]
 \[ g_{9,2}(3) = 3 g_{5,1}(1) + 2 
g_{6,1}(2) + g_{7,1}(3)= 3\cdot 1 
+2\cdot 5 + 13 = 26. \]
 \end{exa}
 
\section{Algebraic results}

The \emph{independence complex} 
$\mathrm{Ind}(G)$ of a graph $G$ is the 
simplicial complex of independent sets 
of $G.$ 
According to Hochster’s formula, if
an edge ideal $I_G$ has a linear resolution then the independence complex $\mathrm{Ind}(G)$ has vanishing
higher homology for all induced subgraphs $H \subseteq G$. In particular $G$ cannot contain induced antiholes
(complements of cycles) of any size larger than 3, implying that $G$ is the complement of
a chordal graph. To obtain our results regarding $\chi$-boundedness, we will investigate which
subgraphs of $G$ are prohibited when we assume the vanishing of a single betti number.

Sometimes the homology of independence complexes of graphs depends on the field of coefficients. As a result of that, the vanishing of a single betti number may depend on the field. But some induced subgraphs are prohibited for all fields, and those are the only ones that we consider. By specifying the coefficient field, one may possibly find stronger $\chi$-bounding functions, but we have no examples of that.

The following three lemmas are 
elementary and essentially well-known:

\begin{lem} \label{indH0} $\dim 
\tilde{H}_0(\mathrm{Ind}(G)) = 0$ if 
and only if the complement of $G$ is 
connected.
\end{lem}
\begin{proof}
The vertices and edges of the simplicial complex $\mathrm{Ind}(G))$ and the complement of the graph $G$ are the same. The dimension of $\tilde{H}_0$ is the number of connected components reduced by one.
\end{proof}

\begin{lem} \label{indSusp} $\dim 
\tilde{H}_{d+1}(\mathrm{Ind}(K_2 \cup 
G)) = \dim 
\tilde{H}_d(\mathrm{Ind}(G))$ for 
$d\geq 0.$
\end{lem}
\begin{proof}
For a simplicial complex $\Delta,$ its \emph{suspension} $\mathrm{susp}(\Delta)$ is constructed by adding two new vertices $u$ and $v$ as well as the faces
\[\{u\} \cup \sigma \textrm{ and }   \{v\} \cup \sigma \]
for all faces $\sigma$ of $\Delta.$ For independence complexes of graphs, this becomes 
\[\mathrm{susp}( \mathrm{Ind}(G)  ) = \mathrm{Ind}(K_2 \cup G) \]
where $K_2$ has vertex set $\{u,v\}.$ By classical algebraic topology, suspension shifts homology by one dimension: 
$\dim \tilde{H}_{d+1}(\mathrm{susp}(\Delta)) = \dim \tilde{H}_d(\Delta).$
\end{proof}

\begin{lem} \label{indDisj} $\dim 
\tilde{H}_{d}(\mathrm{Ind}(K_m + G)) = 
\dim \tilde{H}_d(\mathrm{Ind}(G))$ for 
$d > 0.$
\end{lem}
\begin{proof}
Geometrically, the simplicial complex $\mathrm{Ind}(K_m + G)$ is the same as the simplicial complex $\mathrm{Ind}(G)$ with $m$ extra disjoint points. The homology $\tilde{H}_d$ do not notice the addition of extra disjoint points when $d>0.$
\end{proof}

This is central for this paper:

\begin{lem}\label{central1}
For any integers $d\geq 0$ and $n\geq 2(d+1),$ if
$H \in {\bf B}_{n,d}$ then $\dim 
\tilde{H}_d(\mathrm{Ind}(H)) > 0.$
\end{lem}

\begin{proof} First we consider
the case $d=0.$ The graphs in 
${\bf B}_{n,0}$
are defined by that their maximal 
degree is $n-1.$ Let $v$ be a vertex of 
$H$ of degree $n-1.$ The complement of 
$H$ is disconnected as $v$ is isolated 
in it. By Lemma \ref{indH0}, $\dim 
\tilde{H}_0(\mathrm{Ind}(H)) > 0.$

Now the case $d>0$ by induction. By 
definition
\[ H \in {\bf B}_{n,d} = 
\bigcup_{m=2d}^{n-2} K_{n-m-2}+(K_2\cup 
{\bf B}_{m,d-1}). \]
Pick the $m$ with $H \in 
K_{n-m-2}+(K_2\cup {\bf B}_{m,d-1}).$ 
Then $H = K_{n-m-2}+(K_2\cup H')$ for 
some $H'\in {\bf B}_{m,d-1}.$ By 
induction, $\dim 
\tilde{H}_{d-1}(\mathrm{Ind}(H')) > 0.$ 
By Lemma \ref{indSusp}, 
$\dim 
\tilde{H}_{d}(\mathrm{Ind}(K_2\cup H')) 
> 0.$ Finally, by Lemma \ref{indDisj},
\[
\dim \tilde{H}_d(\mathrm{Ind}(H)) = 
\dim 
\tilde{H}_{d}(\mathrm{Ind}(K_{n-m-2}+(K_
2\cup H')))>0.
\]
\end{proof}

\begin{lem} \label{central2}
If $G$ is a graph with $\dim 
\tilde{H}_{d}(\mathrm{Ind}(H)) = 0$ for 
every induced subgraph $H$ on $n$ 
vertices of $G,$ then $G$ is a ${\bf 
B}_{n,d}$--free graph.
\end{lem}
\begin{proof}
Assume to the contrary that $G$ is not 
a ${\bf B}_{n,d}$--free graph, 
but\linebreak $\dim 
\tilde{H}_d(\mathrm{Ind}(H)) = 0$ for 
every induced subgraph $H$ on $n$ 
vertices of $G.$
Let $H$ be an induced subgraph of  $G$  
on $n$ vertices that is not ${\bf 
B}_{n,d}$--free. All graphs in ${\bf 
B}_{n,d}$ are on $n$ vertices, so $H 
\in {\bf B}_{n,d}.$
By Lemma \ref{central1},  $\dim 
\tilde{H}_d(\mathrm{Ind}(H)) > 0$ and 
we have a contradiction.
\end{proof}

The converse of Lemma~\ref{central2} is not true: The graph $C_5$ is ${\bf B}_{5,1}$--free, as $C_5$ is not in ${\bf B}_{5,1}.$ But $\dim \tilde{H}_{1}(\mathrm{Ind}(C_5)) = 1.$

For an arbitrary ideal $I$ in a ring $S,$ one can encode the successive modules of syzygies—the linear dependencies among generators, the relations among those relations, and so forth. The collection of that data is a \emph{free resolution} and its \emph{length} is the number of times you need relations among relations among relations and so on. For a polynomial ring on $n$ variables, Hilbert's syzygy theorem states that the length is finite and at most $n$. In a well defined sense, minimal resolutions are unique, see Eisenbud \cite{eisenbud1995}. The dimensions of the pieces of the minimal resolution are encoded in the \emph{betti numbers} $\beta_{i,j}(I).$ Those numbers are intrinsically linked to geometrical phenomena of the scheme cut out by $I$ and were introduced to capture intuitive reasoning about intersections between curves in a stringent manner, see Eisenbud~\cite{eisenbud2005}. For edge ideals there is a more direct way to compute their betti numbers, and we take that as its definition. By tensoring with the Koszul resolution, or applying Hochster's formula, as explained by Miller and Sturmfels~\cite{miller2005}, we have that:

\begin{defn}
\[
\beta_{i,j}(I_G) = \sum_{W \in { V(G) 
\choose j } } \dim_{\mathbf 
k}\tilde{H}_{j-i-2}(\mathrm{Ind}(G[W])).
\]
\end{defn}

There is an extensive literature on the betti numbers of edge ideals, for example Katzman \cite{K06}, Hà and Van Tuyl \cite{Ha08} and Woodroofe \cite{W14}.

\begin{thm} \label{mainThm}
The class of graphs with $ 
\beta_{i,j}(I_G) = 0$ is 
$\chi$--bounded by 
$g_{j,j-i-2}(\omega).$
\end{thm}
\begin{proof}
If $ \beta_{i,j}(I_G) = 0 $ then 
$ \dim_{\mathbf 
k}\tilde{H}_{j-i-2}(\mathrm{Ind}(G[W])) 
= 0 $ for all $W \in { V(G) \choose j 
}.$
So, by Lemma \ref{central2}, the graph 
$G$ is a ${\bf B}_{j,j-i-2}$--free 
graph.
By Theorem \ref{bigthm}, the class of  
${\bf B}_{j,j-i-2}$--free graphs is 
$\chi$--bounded by 
$g_{j,j-i-2}(\omega).$
\end{proof}

\begin{cor} \label{mainCor1} The class 
of graphs with $ \beta_{i,j}(I_G) = 0$ 
is $\chi$--bounded by
\[ {\omega-1+2(j-i-2) \choose 2(j-i-2)} 
+  {j-2 \choose 2(j-i-2)+1}. \] 
\end{cor}
\begin{proof} Use Theorem \ref{mainThm} 
and  Proposition \ref{prepp}. 
\end{proof}

\begin{cor} \label{mainCor2} If $G$ is 
a triangle-free graph with 
$\beta_{i,j}(I_G) = 0$ then $\chi(G) 
\leq j-1.$
\end{cor}
\begin{proof} Use Theorem \ref{mainThm} 
and  Proposition \ref{gTriangleFree}. 
\end{proof}

\section{Asymptotically $\chi$-bounded 
graphs}

Reed and Yuditsky \cite{RY25} defined 
an asymptotic version of $\chi$--boundedness, where the coloring condition should be satisfied for almost all graphs, as the number of vertices goes towards infinity.  More formally:
\begin{defn}
A hereditary class $\mathcal{C}$ of 
graphs is \emph{asymptotically 
$\chi$-bounded by~$f$} if
\[
\lim_{n \rightarrow \infty} {\mathbf P} 
 [ \chi(G)\leq f(\omega(G))  \mid 
\textrm{$G$ is an order $n$ graph in 
the class $\mathcal{C}$}] = 1.
\]
\end{defn}
Towards the Gyárfás-Sumner conjecture, 
they proved that for any tree $T,$ the 
class of $T$--free graphs is 
asymptotically $\chi$-bounded by 
$\omega.$

\begin{defn}
A betti number $\beta_{i,j}(I_G)$ is 
\emph{parabolic} if $(j-i)^2 \geq 
j+i+2.$
\end{defn}

This  theorem was proved by Engström 
and Orlich \cite{EO23}.

\begin{thm}\label{eo} Let 
$\beta_{i,j}(I_G)$ be a parabolic betti 
number with $j\geq i+3.$ Then, for 
almost every graph with $ 
\beta_{i,j}(I_G) = 0$ we have:
\begin{itemize}
\item[1.] $\beta_{k,l}(I_G)=0$ whenever 
$l-k=j-i,$ and
\item[2.] parabolic 
$\beta_{k,l}(I_G)>0$ whenever  
$l-k<j-i.$
\end{itemize}
\end{thm}

\begin{cor} \label{asymBound} If  
$(j-i)^2 \geq j+i+2$ and $j-i\geq 3,$ 
then the hereditary class of graphs 
with $ \beta_{i,j}(I_G) = 0$ is 
asymptotically $\chi$-bounded by 
\[ {\omega-1+2(j-i-2) \choose 
2(j-i-2)}. \]
\end{cor}
\begin{proof}
Consider $\beta_{j-i-2,2j-2i-2}(I_G).$ 
We have that $(2j-2i-2)-(j-i-2)=j-i,$ 
so by Theorem \ref{eo} above 
$\beta_{j-i-2,2j-2i-2}(I_G)=0$ for 
almost all $G.$
By Corollary  \ref{mainCor1},
\[
\begin{array}{rcl}
\chi(G) & \leq &  \displaystyle 
{\omega-1+2((2j-2i-2)-(j-i-2)-2) 
\choose 2((2j-2i-2)-(j-i-2)-2)} +  
{2j-2i-2-2 \choose 
2((2j-2i-2)-(j-i-2)-2)+1} \\
& =  & \displaystyle  
{\omega-1+2(j-i-2) \choose 2(j-i-2)} +  
{2j-2i-2-2 \choose 2(j-i-2)+1} \\
& =  & \displaystyle  
{\omega-1+2(j-i-2) \choose 2(j-i-2)} +  
{2j-2i-4 \choose 2j-2i-3} \\
& =  & \displaystyle  
{\omega-1+2(j-i-2) \choose 2(j-i-2)}. \\
\end{array}
\]
\end{proof}
\section{Complexity theoretic results}

There is a long history in studying complexity theoretic results in relation to $\chi$--bounded graph classes, see for example Asplund and Gr\"unbaum \cite{AG60} and Chalermsook and Walczak \cite{CW22}. The goal is to find a 
graph coloring in polynomial time in the number of vertices, or even better than that. We prove a cubic bound
in general and a quadratic for triangle-free graphs.


All graph coloring in this text so far 
builds on repeated applications of 
Lemma~\ref{K2extend}. The main 
complexity problem is that it requires 
one to find a maximum clique, and that 
is notoriously slow.
The first goal of this section is to 
replace the maximum clique in the proof 
of Lemma~\ref{K2extend} by something 
weaker that is faster to find. A clique 
on $K=\{1,2,\ldots, k\}$ in $G$ is 
selected with 
$k=\omega(G)$ to achieve the following 
two facts in the proof:
\begin{itemize}
\item[(1)] There is no $v \in 
V(G)\setminus K$ adjacent to all 
vertices of $K.$
\item[(2)] There is no $u\in K$ 
together with two adjacent vertices $v 
,w \in V(G)\setminus K$ that are 
adjacent to all vertices of $K\setminus 
u.$
\end{itemize}
A clique only satisfying condition (1) 
is maximal.

Andrade, Resende and Werneck  
\cite{arw12}  found a fast heuristic to 
find large independent sets by local 
improvements. As mentioned in their 
paper, it is also valid to find 
cliques, which is the relevant context 
for this paper. Transforming a clique 
by deleting $k-1$ vertices and then 
adding $k$ vertices to get a new larger 
clique is called a 
\emph{$k$--improvement}. If no 
$k$--improvement is possible, the 
clique is \emph{$k$--maximal}. To be 
1-maximal is equivalent to ordinary 
maximal. It is easy to prove that:
\begin{lem}\label{1-imp}
There is a positive constant 
$C_{\mathrm{1-imp}}$ such that given 
any clique in a non-empty graph on $n$ 
vertices and $m$ edges one can find an 
1-improvement (or prove that none 
exist) in $C_{\mathrm{1-imp}}(m+n)$ 
time.
\end{lem}
The heuristic by Andrade et al is fast 
in practice and they also proved that:
\begin{lem}\label{2-imp}
There is a positive constant 
$C_{\mathrm{2-imp}}$ such that given 
any 1-maximal clique in a non-empty 
graph on $n$ vertices and $m$ edges one 
can find an 2-improvement (or prove 
that none exist) in 
$C_{\mathrm{2-imp}}(m+n)$ time.
\end{lem}
Putting them together we may find a 
2-maximal clique.
\begin{lem}\label{comp1}
There is a positive constant 
$C_{\mathrm{2-max}}$ such that given 
any non-empty graph on $n$ vertices and 
$m$ edges with maximum clique size 
$\omega$ one may find a 2-maximal 
clique in 
$C_{\mathrm{2-max}}\omega(m+n)$ time.
\end{lem}
\begin{proof}
Start with an empty clique and use 
Lemma~\ref{1-imp} and  
Lemma~\ref{2-imp} until the clique is 
2-maximal. The number of applications 
of them is in total at most $\omega,$ 
as that is the maximum size of a clique 
in $G.$ Set $C_{\mathrm{2-max}} = 
\max(C_{\mathrm{1-imp}} 
,C_{\mathrm{2-imp}}).$
\end{proof}
\begin{lem} For a set ${\bf H}$ of graphs, if
\begin{itemize} 
\item[-]  the class of ${\bf H}$--free 
graphs is  $\chi$-bounded by $f_{\bf 
H},$ and
\item[-]  there exist a positive 
constant $C_{\bf H}$ such that given 
any non-empty graph~$G$ on $n$ vertices 
and $m$ edges with maximum clique size 
$\omega$ one may find a
$f_{\bf H}(\omega)$--coloring of $G$ in 
$C_{\bf H}\omega(m+n)$ time,
\end{itemize}
then
\begin{itemize}
\item[-] the class of  $K_2 \cup {\bf H 
 } $--free graphs is $\chi$-bounded by
\[
f_{K_2 \cup {\bf H  } }(\omega) = 
\sum_{k=1}^{\omega}  (\omega-k+1)  
f_{{\bf H  } }(k),
\]
\item[-] and there exist a positive 
constant $C_{K_2 \cup {\bf H  } }$ such 
that given any non-empty graph~$G$ on 
$n$ vertices and $m$ edges with maximum 
clique size $\omega$ one may find a
$f_{K_2 \cup {\bf H  } 
}(\omega)$--coloring of $G$ in $C_{K_2 
\cup {\bf H  } }\omega(m+n)$ time.
\end{itemize}
\end{lem}
\begin{proof}
The statements on $\chi$--boundedness 
are all in Lemma~\ref{K2extend}. The 
only new thing in Lemma~\ref{comp1} is 
the complexity theoretic statements. We 
follow the proof of 
Lemma~\ref{K2extend} and work out the 
time used. This starts off with finding 
a 2-maximal clique instead of a maximum 
clique. It is found in time 
$C_{\mathrm{2-max}}\omega(m+n)$ 
according to Lemma~\ref{comp1}.

Let $\tilde\omega$ be the cardinality 
of the 2-maximal clique. The next step 
is to find the sets $C_{i,j}$ for 
$1\leq i < j \leq \tilde\omega$ and 
$D_i$ for $1\leq i \leq \tilde\omega.$ 
It is straightforward to do that in 
$C'\omega(m+n)$ time, for some constant 
$C'.$ Actually a minor adaptation of 
the algorithm to find the 2-maximal 
clique provides this information almost 
for free.

Each $D_i$ is independent, and coloring 
all of them is linear in the number of 
vertices, say $C''n.$

Now let $n_{i,j}$ and $m_{i,j}$ be the 
number of vertices and edges in each 
$C_{i,j},$ respectively. Coloring all 
of them is done in
\[
\sum_{1\leq i < j \leq \tilde\omega}  
C_{\bf H}\omega(m_{i,j}+n_{i,j}) \leq 
C_{\bf H}\omega(m+n)
\]
time. In total at most $C_{K_2 \cup 
{\bf H  } }\omega(m+n)$ time is used, 
where essentially one may set $C_{K_2 
\cup {\bf H  } } = C_{\mathrm{2-max}} + 
C' + C'' + C_{\bf H}.$
\end{proof}

A minor detail above is that if  
$\tilde\omega<\omega$ then a better 
chromatic bound is actually achieved:
\[ \sum_{k=1}^{\tilde \omega}  
(\omega-k+1)  f_{{\bf H  } }(k) \]

In the same way as Lemma~\ref{comp1} is 
the complexity version of 
Lemma~\ref{K2extend}, the following 
Theorem~\ref{comp2} is the complexity 
version of Theorem~\ref{bigthm}.
\begin{thm}\label{comp2}
For any integers $d\geq 0$ and $n\geq 
2(d+1)$ there is a constant $C$ such 
that every ${\bf B}_{n,d}$--free graph 
$G$ on $n'$ vertices and $m'$ edges is 
$g_{n,d}(\omega(G))$--colorable in 
$C\omega(G)(m'+n')$ time.
\end{thm}
\begin{proof}
The existence of the coloring is the 
content of Theorem~\ref{bigthm}. To be 
proved is the complexity result. In the 
proof of Theorem~\ref{bigthm} the 
initial part is that graphs of constant 
maximal degree are colored optimally. 
That can be done in $C'(m'+n')$ time 
using a first-fit coloring. After that, 
a version of Lemma~\ref{K2extend} with 
refined bookkeeping is applied $d$ 
times. From Lemma~\ref{comp1} we know 
that each application is done faster 
than $C''\omega(G)(m'+n')$ for some 
$C''.$ Set $C=C'+dC''.$
\end{proof}

We also get complexity versions of 
Corollaries \ref{mainCor1} and 
\ref{mainCor2}.

Recall that in discussing betti numbers 
$\beta_{i,j}(I_G)$ of edge ideals, we 
assume from now that $j\leq 2i+2,$ 
because otherwise they are trivially 
zero.

\begin{cor}\label{n3} Let $i$ and $j$ 
be positive integers. Any graph $G$ on 
$n$ vertices and clique number $\omega$ 
with $\beta_{i,j}(I_G) = 0$ has 
chromatic number
\[ \chi \leq {\omega-1+2(j-i-2) \choose 
2(j-i-2)} +  {j-2 \choose 2(j-i-2)+1} \]
and such a coloring may be found in 
$O(n^3)$ time. 
\end{cor}
\begin{proof}
The coloring result in Corollary 
\ref{mainCor1} follows from 
Theorem~\ref{bigthm}, whose complexity 
version is Theorem~\ref{comp2} above. 
According to that, the coloring is 
found in $C\omega(m+n)$ time where $m$ 
is the number of edges in $G$ and 
$\omega$ is the clique number. From 
$m\leq n^2$ and $\omega\leq n$ it 
follows that $O(n^3)$ time is used.
\end{proof}
In the same manner:
\begin{cor}\label{n2} Let $i$ and $j$ 
be positive integers. Any triangle-free 
graph $G$ on $n$ vertices with 
$\beta_{i,j}(I_G) = 0$ 
is $(j-1)$--colorable and such a 
coloring may be found in $O(n^2)$ time.
\end{cor}

\section{Some further considerations and an outlook}

Instead of the ${\bf B}_{n,d}$ sets defined in Section 3, one may consider other sets of forbidden graphs as they do not contain all possible obstructions to vanishing syzygies. For example, $C_5$ is a possible addition to ${\bf B}_{5,1}.$ After adding $C_5$ to ${\bf B}_{5,1},$ the recursion
\[{\bf B}_{n,d} = \bigcup_{m=2d}^{n-2} K_{n-m-2}+(K_2\cup {\bf B}_{m,d-1}) \] 
is still valid and may be employed to extend ${\bf B}_{n,d}$ for $d\geq 2$ and $n\geq 5+2(d-1).$
Moreover, the complement of any cycle $C_n$ may be added to  ${\bf B}_{n,1}$ for $n>5.$

Another aspect would be to reduce the size of ${\bf B}_{n,d}$ and sacrifice the bound achieved to get an easier proof. There are at least two natural candidates in that direction. The first one is to set ${\bf B}_{n,0}' = {\bf B}_{n,0}$ and
\[{\bf B}_{n,d}' = K_2\cup {\bf B}_{n-2,d-1}'. \]
The ${\bf B}_{n,d}'$--free graphs are $\chi$--bounded by a polynomial of the same degree as the ${\bf B}_{n,d}$--free graphs, but with much worse coefficients. On the other hand, Theorem~\ref{bigthm} may be replaced with the easier Corollary~\ref{corBoundingSimple}. Completely disregarding the coefficients and only aiming to preserve the degree of the $\chi$--bounding polynomials while simplifying the argument even more, one may replace Corollary~\ref{corBoundingSimple} by the basic inductive step in Wagon's proof of that $pK_2$--free graphs are $\chi$--bounded by a polynomial of degree $2p-2.$

A further reduction would be to start off with ${\bf B}_{n,0}'' = \{ \textrm{the star on $n$ vertices} \}$ and continue with ${\bf B}_{n,d}'' = K_2\cup {\bf B}_{n-2,d-1}''.$ The upside is that each set only contains one graph:
\[ {\bf B}_{n,d}'' =  \{ (d+1)K_2 \cup \textrm{the star on $n-2(d+1)$ vertices} \}. \]
The main downside of this construction is that a star on $m$ vertices is $\chi$--bounded by $k\frac{\omega^{m-2}}{\log^{m-3} (\omega)}$ for some constant $k.$ This shows that ${\bf B}_{n,0}''$--free graphs are $\chi$--bounded by $k\frac{\omega^{n-2}}{\log^{n-3} (\omega)},$ while the ${\bf B}_{n,0}'$--free graphs are $\chi$--bounded by $n.$\linebreak  This factor propagates to that
the ${\bf B}_{n,d}''$--free graphs are $\chi$--bounded by\linebreak $\frac{k}{n-2(d+1)}\frac{\omega^{n-2(d+1)-2}}{\log^{n-2(d+1)-3} (\omega)}$ times whatever the ${\bf B}_{n,d}'$--free graphs are $\chi$--bounded by. Thus, the ${\bf B}_{n,d}''$--free graphs also proves polynomially $\chi$--bounded, but by a worse degree.

A disjoint union of stars is called a \emph{star-forest}. It was proved by Scott, Seymour and Spirkl \cite{sss} that for any star-forest $H,$ the $H$--free graphs are polynomially $\chi$--bounded. They did not make any effort in optimizing degrees of the polynomials, and it does not seem straight-forward to use their result right off to improve on our results. What perhaps would be feasible is to exclude many star-forests at once. Define
\[
{\bf B}_{n,d}^\dagger = \{ \textrm{star-forests on $d+1$ components and $n$ vertices}  \}.
\]
Clearly the ${\bf B}_{n,d}^\dagger$--free graphs are polynomially $\chi$--bounded as every forbidden graph is a star-forest. It's an interesting research problem to find a better chromatic bound for the ${\bf B}_{n,d}^\dagger$--free graphs. The independence complex of a star-forest on $d+1$ components collapses onto the $d$--dimensional boundary of a cross polytope, providing the syzygy obstruction.

\subsubsection*{Acknowledgements}
The author would like to thank Ignacio García-Marco, Philippe Gimenez and Kolja Knauer for their comments and the referees for their excellent advice on both content and exposition.

\end{document}